\documentclass[11pt]{article}
\usepackage{geometry,amsthm,amssymb,amsmath,enumerate,float,cite}
\newtheorem{theorem}{Theorem}

\newtheorem{claim}{Claim}

\allowdisplaybreaks

\title{On the extremal graphs for degenerate subsets,\\
dynamic monopolies, and partial incentives}
\author{S. Ehard \and  D. Rautenbach}
\date{}

\begin{document}


\maketitle

\begin{center}
Institut f\"{u}r Optimierung und Operations Research,
Universit\"{a}t Ulm, Ulm, Germany,
\{\texttt{stefan.ehard, dieter.rautenbach}\}\texttt{@uni-ulm.de}\\[3mm]
\end{center}

\begin{abstract}
The famous lower bound $\alpha(G)\geq \sum_{u\in V(G)}\frac{1}{d_G(u)+1}$ 
on the independence number $\alpha(G)$ of a graph $G$
due to Caro and Wei 
is known to be tight if and only if the components of $G$ are cliques,
and has been generalized several times
in the context of large degenerate subsets and small dynamic monopolies.
We characterize the extremal graphs for a generalization 
due to Ackerman, Ben-Zwi, and Wolfovitz.
Furthermore, we give a simple proof of a related bound 
concerning partial incentives due to 
Cordasco, Gargano, Rescigno, and Vaccaro,
and also characterize the corresponding extremal graphs.
\end{abstract}
{\small
\begin{tabular}{lp{12.5cm}}
\textbf{Keywords:} & 
independent set; 
degenerate set;
dynamic monopoly;
target set;
partial incentives
\end{tabular}
}

\section{Introduction}

We consider finite, simple, and undirected graphs, and use standard terminology.
Throughout this paper, let $G$ be a graph,
and let $c:V(G)\to \mathbb{R}_{>0}$ and $\kappa:V(G)\to \mathbb{Z}$ 
be two functions with $0\leq \kappa(u)\leq d_G(u)$ for every vertex $u$ of $G$,
where $V(G)$ denotes the vertex set of $G$, and $d_G(u)$ denotes the degree of a vertex $u$ in $G$.
For a set $I$ of vertices of $G$,
let the {\it $c$-weight} of $I$ be $c(I)=\sum_{u\in I}c(u)$.
The set $I$ is {\it $\kappa$-degenerate} in $G$
if there is a linear ordering $u_1,\ldots,u_k$ of the vertices in $I$
such that $u_i$ has at most $\kappa(u_i)$ neighbors in $\{ u_j:j\in [i-1]\}$ for every $i\in [k]$,
where $[n]$ denotes the set of positive integers at most $n$ for every integer $n$.
Note that a set of vertices of $G$ is {\it independent} exactly if it is $0$-degenerate.
Therefore, if $\alpha(G,c,\kappa)$ denotes the maximum $c$-weight of a $\kappa$-degenerate set of vertices of $G$,
then $\alpha(G,1,0)$ is the well-known {\it independence number} $\alpha(G)$ of $G$.

For every graph $G$, Caro \cite{ca} and Wei \cite{we} showed
\begin{eqnarray}\label{ecw}
\alpha(G) & = & \alpha(G,1,0) \geq \sum\limits_{u\in V(G)}\frac{1}{d_G(u)+1}.
\end{eqnarray}
For a fixed non-negative integer $d$,
Alon, Kahn, and Seymour \cite{alkase} extended (\ref{ecw}) by showing
\begin{eqnarray}\label{eaks}
\alpha(G,1,d) & \geq & \sum\limits_{u\in V(G)}\frac{\min\{ d_G(u),d\}+1}{d_G(u)+1}.
\end{eqnarray}
The dual notion of a degenerate set of vertices is the notion of a dynamic monopoly or target set \cite{ch,acbewo,re}.
More precisely, given $G$ and $\kappa$ as above, 
if the function $\tau:V(G)\to\mathbb{Z}$ is such that $d_G(u)=\tau(u)+\kappa(u)$ for every vertex $u$ of $G$,
then a set $I$ of vertices of $G$ is $\kappa$-degenerate in $G$
if and only if $V(G)\setminus I$ is a {\it dynamic monopoly} or {\it target set} in $G$ with {\it threshold function} $\tau$.
This duality generalizes the well-known duality
between independent sets and vertex covers.
Using this duality, the following generalization of (\ref{ecw}) and (\ref{eaks}) 
is an equivalent formulation of a result due to Ackerman, Ben-Zwi, and Wolfovitz \cite{acbewo} (cf.~also Reichman \cite{re}).
\begin{eqnarray}\label{eabw}
\alpha(G,1,\kappa) & \geq & \sum\limits_{u\in V(G)}\frac{\kappa(u)+1}{d_G(u)+1}.
\end{eqnarray}
Cordasco, Gargano, Rescigno, and Vaccaro \cite{cogareva} gave an algorithmic proof of the following weighted extension of (\ref{eabw}).
\begin{eqnarray}\label{ecgrv}
\alpha(G,c,\kappa) & \geq & \sum\limits_{u\in V(G)}\frac{c(u)(\kappa(u)+1)}{d_G(u)+1}.
\end{eqnarray}
The simple probabilitstic proofs \cite{alsp}
known for (\ref{ecw}) and (\ref{eabw}) 
also work for (\ref{ecgrv}).
In fact, if $u_1,\ldots,u_n$ 
is a linear ordering of the vertices of $G$
chosen uniformly at random, then 
$$I=\left\{ u_i:i\in [n]\mbox{ and }\Big| N_G(u_i)\cap \{ u_j:j\in [i-1]\}\Big|\leq \kappa(u_i)\right\}$$
is $\kappa$-degenerate, 
and the right hand side of (\ref{ecgrv}) equals $\mathbb{E}[c(I)]$,
that is, the first moment method implies (\ref{ecgrv}).

Motivated by a scenario involving partial incentives,
Cordasco et al.~\cite{cogareva} consider --- the equivalent dual of ---
the following problem for $G$ and $\kappa$ as above:
\begin{eqnarray}\label{epi}
\beta(G,1,\kappa):=\min\Big\{
\iota(V(G)):\mbox{$V(G)$ is $(\kappa+\iota)$-degenerate for }
\iota:V(G)\to\mathbb{Z}_{\geq 0}\Big\},
\end{eqnarray}
that is, the minimum total pointwise increase $\iota(V(G))=\sum_{u\in V(G)}\iota(u)$ 
of the function $\kappa$ such that the entire vertex set of 
$G$ becomes $\kappa'$-degenerate for the new function $\kappa'=\kappa+\iota$,
or, equivalently, the empty set becomes a dynamic monopoly 
for the threshold function $d_G-\kappa'$.
A natural weighted version of (\ref{epi}) is
\begin{eqnarray}\label{ewpi}
\beta(G,c,\kappa):=\min\left\{
\sum\limits_{u\in V(G)}c(u)\iota(u):\mbox{$V(G)$ is $(\kappa+\iota)$-degenerate for }
\iota:V(G)\to\mathbb{Z}_{\geq 0}\right\}.
\end{eqnarray}
For $G$ and $\kappa$ as above,
Cordasco et al.~\cite{cogareva} gave an involved algorithmic proof of the following inequality (cf.~Theorem 6 in \cite{cogareva}).
\begin{eqnarray}\label{epi2}
\beta(G,1,\kappa) & \leq & 
\sum\limits_{u\in V(G)}
\frac{\Big(d_G(u)-\kappa(u)\Big)\Big(d_G(u)-\kappa(u)+1\Big)}{2(d_G(u)+1)}.
\end{eqnarray}
Our first contribution is a simple probabilistic proof 
of a weighted generalization of (\ref{epi2}). 
\begin{theorem}\label{theorem0}
If $G$ is a graph, 
and $c:V(G)\to \mathbb{R}_{>0}$ and $\kappa:V(G)\to \mathbb{Z}$
are such that $0\leq \kappa(u)\leq d_G(u)$ for every vertex $u$ of $G$,
then 
\begin{eqnarray}\label{ewpi2}
\beta(G,c,\kappa) & \leq & 
\sum\limits_{u\in V(G)}
\frac{c(u)\Big(d_G(u)-\kappa(u)\Big)\Big(d_G(u)-\kappa(u)+1\Big)}{2(d_G(u)+1)}.
\end{eqnarray}
\end{theorem}
\begin{proof}
If $u_1,\ldots,u_n$ is a linear ordering of the vertices of $G$
chosen uniformly at random, 
and
$$\iota(u_i)=\max\Big\{0,\big| N_G(u_i)\cap \{ u_j:j\in [i-1]\}\big|-\kappa(u_i)\Big\},$$ 
then $\iota(u_i)\in \mathbb{Z}_{\geq 0}$ for every $i\in [n]$, and
$V(G)$ is $\kappa'$-degenerate for $\kappa'=\kappa+\iota$.

Since 
\begin{eqnarray*}
\mathbb{E}[\iota(u_i)] 
& = & \sum\limits_{\ell=0}^{d_G(u_i)}
\mathbb{P}\Big[\big| N_G(u_i)\cap \{ u_j:j\in [i-1]\}\big|=\ell\Big]
\cdot\max\Big\{0,\ell-\kappa(u_i)\Big\}\\
& = & \sum\limits_{\ell=0}^{d_G(u_i)}
\frac{1}{d_G(u_i)+1}
\cdot\max\Big\{0,\ell-\kappa(u_i)\Big\}\\
& = & \sum\limits_{\ell=1}^{d_G(u_i)-\kappa(u_i)}
\frac{\ell}{d_G(u_i)+1}\\
& = & \frac{1}{d_G(u_i)+1}{d_G(u_i)-\kappa(u_i)+1\choose 2}, 
\end{eqnarray*}
we obtain, by linearity of expectation,
\begin{eqnarray*}
\beta(G,c,\kappa) & \leq & \mathbb{E}\left[\sum\limits_{i\in [n]}c(u_i)\cdot \iota(u_i)\right]\\
&=&\sum\limits_{i\in [n]}c(u_i)\cdot \mathbb{E}[\iota(u_i)]\\
&\leq &\sum\limits_{i\in [n]}
\frac{c(u_i)}{d_G(u_i)+1}{d_G(u_i)-\kappa(u_i)+1\choose 2},
\end{eqnarray*}
which completes the proof.
\end{proof}
As our main contribution
we characterize the extremal graphs for (\ref{ecgrv}) and (\ref{ewpi2}).
Our results generalize the well-known fact that (\ref{ecw}) is achieved with 
equality if and only if $G$ is the disjoint union of cliques.

\section{Extremal graphs for (\ref{ecgrv}) and (\ref{ewpi2})}

While probabilistic arguments lead to simple and short proofs for (\ref{ecgrv}) and (\ref{ewpi2}), the extremal graphs can more easily be extracted 
from proofs mimicking greedy algorithms.
Since (\ref{ecgrv}) and (\ref{ewpi2}) are both linear with respect to the components,
it suffices to characterize the connected extremal graphs.

\begin{theorem}\label{theorem1}
If $G$ is a connected graph, 
and $c:V(G)\to \mathbb{R}_{>0}$ and $\kappa:V(G)\to \mathbb{Z}$ 
are such that $0\leq \kappa(u)\leq d_G(u)$ for every vertex $u$ of $G$,
then (\ref{ecgrv}) holds with equality if and only if 
\begin{enumerate}[(i)]
\item either $\kappa(u)=d_G(u)$ for every vertex $u$ of $G$,
\item or 
$G$ is a clique, 
and $c$ and $\kappa$ are constant on $V(G)$.
\end{enumerate}
\end{theorem}
\begin{proof} 
Let $f(G,c,\kappa)$ denote the right hand side of (\ref{ecgrv}).
Clearly, if (i) or (ii) hold, then (\ref{ecgrv}) holds with equality.
We call a triple $(G,c,\kappa)$ {\it extremal} 
if $\alpha(G,c,\kappa)=f(G,c,\kappa)$.
Now, let $(G,c,\kappa)$ be extremal.
Note that we do not yet assume that $G$ is connected.

\begin{claim}\label{claim1}
For every vertex $u$ of $G$,
$$c(u)=\frac{c(u)(\kappa(u)+1)}{d_G(u)+1}
+\sum\limits_{v\in N_G(u)}
\frac{c(v)(\kappa(v)+1)}{d_G(v)+1}
-\sum\limits_{v\in N_G(u)}
\frac{c(v)\kappa(v)}{d_G(v)}.$$
\end{claim}
\begin{proof}[Proof of Claim \ref{claim1}]
Since, for every function $h:V(G)\to \mathbb{R}$,
we have
\begin{eqnarray}\label{eh}
\sum\limits_{u\in V(G)}\sum\limits_{v\in N_G(u)}h(v)
=\sum\limits_{u\in V(G)}h(u)d_G(u),
\end{eqnarray}
we obtain
\begin{eqnarray*}
&&\sum\limits_{u\in V(G)}\left(
-c(u)+\frac{c(u)(\kappa(u)+1)}{d_G(u)+1}
+\sum\limits_{v\in N_G(u)}
\frac{c(v)(\kappa(v)+1)}{d_G(v)+1}
-\sum\limits_{v\in N_G(u)}
\frac{c(v)\kappa(v)}{d_G(v)}
\right)\\
& = & \sum\limits_{u\in V(G)}\left(
-c(u)+\frac{c(u)(\kappa(u)+1)}{d_G(u)+1}
+\sum\limits_{v\in N_G(u)}
\frac{c(v)(d_G(v)-\kappa(v))}{d_G(v)(d_G(v)+1)}
\right)\\
& \stackrel{(\ref{eh})}{=} & \sum\limits_{u\in V(G)}
\underbrace{\left(
-c(u)+\frac{c(u)(\kappa(u)+1)}{d_G(u)+1}
+\frac{c(u)(d_G(u)-\kappa(u))}{(d_G(u)+1)}
\right)}_{=0}\\
& = & 0.
\end{eqnarray*}
Hence, if the statement of the claim does not hold,
then there is a vertex $u'$ of $G$ with 
\begin{eqnarray}
c(u')>\frac{c(u')(\kappa(u')+1)}{d_G(u')+1}
+\sum\limits_{v\in N_G(u')}
\frac{c(v)(\kappa(v)+1)}{d_G(v)+1}
-\sum\limits_{v\in N_G(u')}
\frac{c(v)\kappa(v)}{d_G(v)}.\label{eb0}
\end{eqnarray}
Let 
\begin{eqnarray}
N' & = & \big\{ v\in N_G(u'):\kappa(v)=0\big\},\nonumber\\
V' & = & V(G)\setminus \big(\{ u'\}\cup N'\big),\nonumber\\
G' & = & G\left[V'\right],\label{eb3}\\
c' & = & c\mid_{V'}\mbox{, and}\nonumber\\
\kappa' &:& V'\to \mathbb{Z}:
v\mapsto 
\begin{cases}
\kappa(v)-1 &\mbox{, if $v\in V'\cap N_G(u')$, and}\\
\kappa(v) & \mbox{, if $v\in V'\setminus N_G(u')$.}
\end{cases}\nonumber
\end{eqnarray}
Note that 
$$\frac{c(v)\kappa(v)}{d_G(v)}=
\begin{cases}
\frac{c(v)(\kappa'(v)+1)}{d_{G'}(v)+1}&\mbox{, if $v\in V'\cap N_G(u')$, and}\\
0&\mbox{, if $v\in N'$.}
\end{cases}$$
By construction, $0\leq \kappa'(u)\leq d_{G'}(u)$ for every vertex $u$ of $G'$, 
and adding $u'$ to a $\kappa'$-degenerate set of vertices of $G'$
yields a $\kappa$-degenerate set of vertices of $G$.
This implies the contradiction
\begin{eqnarray}
\alpha(G,c,\kappa) 
& \geq & \alpha(G',c',\kappa')+c(u')\nonumber\\
& \stackrel{(\ref{ecgrv})}{\geq} & f(G',c',\kappa')+c(u')\nonumber\\
& \stackrel{(\ref{eb0})}{>} &
f(G',c',\kappa')+\frac{c(u')(\kappa(u')+1)}{d_G(u')+1}
+\sum\limits_{v\in N_G(u')}
\frac{c(v)(\kappa(v)+1)}{d_G(v)+1}
-\sum\limits_{v\in N_G(u')}
\frac{c(v)\kappa(v)}{d_G(v)}\nonumber\\
& \geq & f(G,c,\kappa),\label{e2}
\end{eqnarray}
which completes the proof of the claim.

Note that if the final inequality (\ref{e2}) in the above inequality chain 
holds with equality,
then $G$ contains no edges between $N'$ and $V'$.
In fact, if some vertex $w$ in $V'\setminus N_G(u')$ has a neighbor in $N'$,
then the contribution $\frac{c(v)(\kappa(v)+1)}{d_{G'}(v)+1}$ 
of $v$ to $f(G',c',\kappa')$
is larger than its contribution $\frac{c(v)(\kappa(v)+1)}{d_G(v)+1}$ 
to $f(G,c,\kappa)$,
and, 
if some vertex $w$ in $V'\cap N_G(u')$ has a neighbor in $N'$,
then the contribution $\frac{c(v)\kappa(v)}{d_{G'}(v)+1}$ 
of $v$ to $f(G',c',\kappa')$
is larger than the subtracted term $\frac{c(v)\kappa(v)}{d_G(v)}$.
\end{proof}
We say that a vertex $u$ of $G$ is {\it initial}
if there is a $\kappa$-degenerate set $I$ 
of $c$-weight $\alpha(G,c,\kappa)$
such that 
there is a linear ordering $u_1,\ldots,u_k$ of the vertices in $I$
such that $u=u_1$, and
$u_i$ has at most $\kappa(u_i)$ neighbors in $\{ u_j:j\in [i-1]\}$ 
for every $i\in [k]$.
\begin{claim}\label{claim2}
Let $u'$ be any vertex of $G$,
and let $N'$, $V'$, $G'$, $c'$, and $\kappa'$ 
be as in (\ref{eb3}).

The vertex $u'$ is initial,
$(G',c',\kappa')$ is extremal,
and there are no edges between $N'$ and $V'$.
\end{claim}
\begin{proof}[Proof of Claim \ref{claim2}]
Since adding $u'$ to a $\kappa'$-degenerate set of vertices of $G'$
yields a $\kappa$-degenerate set of vertices of $G$,
we obtain
\begin{eqnarray*}
\alpha(G,c,\kappa) 
& \geq & \alpha(G',c',\kappa')+c(u')\\
& \stackrel{(\ref{ecgrv})}{\geq} & f(G',c',\kappa')+c(u')\\
& \stackrel{\rm Claim\,\,\ref{claim1}}{=} &
f(G',c',\kappa')+\frac{c(u')(\kappa(u')+1)}{d_G(u')+1}
+\sum\limits_{v\in N_G(u')}
\frac{c(v)(\kappa(v)+1)}{d_G(v)+1}
-\sum\limits_{v\in N_G(u')}
\frac{c(v)\kappa(v)}{d_G(v)}\\
& \geq & f(G,c,\kappa).
\end{eqnarray*}
Since $\alpha(G,c,\kappa)=f(G,c,\kappa)$,
equality holds throughout this inequality chain.
Since $\alpha(G,c,\kappa)=\alpha(G',c',\kappa')+c(u')$,
it follows that $u'$ is initial.
Since $\alpha(G',c',\kappa')=f(G',c',\kappa')$,
it follows that $(G',c',\kappa')$ is extremal.
As noted at the end of the proof of Claim \ref{claim1},
equality in the last inequality of the above inequality chain implies 
that there are no edges between $N'$ and $V'$.
\end{proof}

\begin{claim}\label{claim3}
If $G$ is connected, then there is some $f_0\in\mathbb{R}_{\geq 0}$
with 
\begin{eqnarray}\label{e3}
f_0=\frac{c(u)(d_G(u)-\kappa(u))}{d_G(u)(d_G(u)+1)}
\end{eqnarray}
for every vertex $u$ of $G$.
\end{claim}
\begin{proof}[Proof of Claim \ref{claim3}]
If $h(u)$ denotes the right hand side of (\ref{e3}),
then Claim \ref{claim1} implies that 
$h(u)=\frac{1}{d_G(u)}\sum_{v\in N_G(u)}h(v)$ for every vertex $u$ of $G$,
that is, the $h$-value of every vertex equals 
the average $h$-value of its neighbors.
Since $G$ is connected, it follows that $h$ is constant within $V(G)$.
\end{proof}
We have shown Claims \ref{claim1}, \ref{claim2}, and \ref{claim3}
for every extremal triple $(G,c,\kappa)$.
For the rest of the proof, we proceed by contradiction,
and assume that the extremal triple $(G,c,\kappa)$
is a counterexample to the statement of the theorem
such that the order $n$ of $G$ is minimum.
Trivially, we have $n\geq 2$.
If $f_0=0$, then $c(u)>0$ implies $\kappa(u)=d_G(u)$ 
for every vertex $u$ of $G$,
that is, (i) holds.
By the choice of $(G,c,\kappa)$, we obtain $f_0>0$, 
which implies $\kappa(u)<d_G(u)$ for every vertex $u$ of $G$.
If $n=2$, then this implies $\kappa(u)=0$ for every vertex $u$ of $G$, which, by (\ref{e3}), implies that $c$ is constant,
that is, (ii) holds.
Hence, the choice of $(G,c,\kappa)$ implies $n\geq 3$.

\begin{claim}\label{claim4}
$G$ is a clique.
\end{claim}
\begin{proof}[Proof of Claim \ref{claim4}]
Suppose, for a contradiction, that $G$ is not a clique.
This implies that $G$ has a vertex $u'$
such that $G-u'$ is connected,
and $u'$ is not universal, that is, 
$N_G(u')\not=V(G)\setminus \{ u'\}$.
Let $N'$, $V'$, $G'$, $c'$, and $\kappa'$ 
be as in (\ref{eb3}).
Since $u'$ is not universal, and $N'\subseteq N_G(u')$,
the set $V'\setminus N_G(u')$ contains a vertex $v'$,
in particular, the set $V'$ is not empty.
Since $G$ and $G-u'$ are connected, and,
by Claim \ref{claim2}, 
there are no edges between $N'$ and $V'$, 
we obtain that 
the set $N'$ is empty, 
the graph $G'$ equals $G-u'$, which is connected, and
the set $V'\cap N_G(u')$ contains a vertex $w'$.
By Claim \ref{claim2},
the triple $(G',c',\kappa')$ is extremal,
which, by Claim \ref{claim3}, implies
the existence of some $f_0'\in\mathbb{R}_{\geq 0}$
with 
\begin{eqnarray}\label{e4}
f_0'=\frac{c(u)(d_{G'}(u)-\kappa'(u))}{d_{G'}(u)(d_{G'}(u)+1)}
\end{eqnarray}
for every vertex $u$ of $G'$.
Using the definition of $G'$ and $c'$, we obtain the contradiction
\begin{eqnarray*}
f_0-f_0'
& \stackrel{(\ref{e3}),(\ref{e4})}{=} & 
\frac{c(w')(d_G(w')-\kappa(w'))}{d_G(w')(d_G(w')+1)}
-\frac{c'(w')(d_{G'}(w')-\kappa'(w'))}{d_{G'}(w')(d_{G'}(w')+1)}\\
& = & 
\frac{c(w')(d_G(w')-\kappa(w'))}{d_G(w')(d_G(w')+1)}
-\frac{c(w')\big((d_G(w')-1)-(\kappa(w')-1)\big)}{(d_G(w')-1)d_G(w')}
<0\mbox{ and}\\[3mm]
f_0-f_0'  
& \stackrel{(\ref{e3}),(\ref{e4})}{=} & 
\frac{c(v')(d_G(v')-\kappa(v'))}{d_G(v')(d_G(v')+1)}
-\frac{c'(v')(d_{G'}(v')-\kappa'(v'))}{d_{G'}(v')(d_{G'}(v')+1)}
=0,
\end{eqnarray*}
which completes the proof of the claim.
\end{proof}
We are now in a position to derive a final contradiction.
Let $u'$ be any vertex of $G$,
and let $N'$, $V'$, $G'$, $c'$, and $\kappa'$ 
be as in (\ref{eb3}).
By Claim \ref{claim2}, there are no edges between $N'$ and $V'$,
which, by Claim \ref{claim4}, implies that either $N'$ or $V'$ is empty.
If $V'$ is empty, then $\kappa$ is constant on $V(G)\setminus \{ u'\}$.
If $N'$ is empty, then
$\kappa(u)<d_G(u)$ for every vertex $u$ of $G$
implies that 
$\kappa'(u)<d_{G'}(u)$ for every vertex $u$ of $G'$.
Since $(G',c',\kappa')$ is extremal,
the choice of $(G,c,\kappa)$ implies 
that $\kappa'$ is constant on $V'$.
Since $u'$ is adjacent to all vertices in $V'$,
the definition of $\kappa'$ implies that 
$\kappa$ is constant on $V(G)\setminus \{ u'\}$.
Altogether, we obtain that 
$\kappa$ is constant on $V(G)\setminus \{ u'\}$
for every vertex $u'$ of $G$.
Since $G$ has at least $3$ vertices, 
this actually implies that 
$\kappa$ is constant on $V(G)$.
By (\ref{e3}), and Claim \ref{claim4},
it follows that also $c$ is constant on $V(G)$,
that is, (ii) holds.
This final contradiction completes the proof. 
\end{proof}
The statement and the proof of the following result 
is quite similar to the statement and the proof of Theorem \ref{theorem1}.
There are nevertheless 
several small yet subtle and important differences,
which we will point out during the proof.

\begin{theorem}\label{theorem2}
If $G$ is a connected graph, 
and $c:V(G)\to \mathbb{R}_{>0}$ and $\kappa:V(G)\to \mathbb{Z}$ 
are such that $0\leq \kappa(u)\leq d_G(u)$ for every vertex $u$ of $G$,
then (\ref{ewpi2}) holds with equality if and only if 
\begin{enumerate}[(i)]
\item either $\kappa(u)=d_G(u)$ for every vertex $u$ of $G$,
\item or $c$ is constant on $V(G)$, and $\kappa(u)=0$ for every vertex $u$ of $G$,
\item or 
$G$ is a clique, 
$c$ and $\kappa$ are constant on $V(G)$,
and $0<\kappa<d_G$.
\end{enumerate}
\end{theorem}
\begin{proof} 
Let $g(G,c,\kappa)$ denote the right hand side of (\ref{ewpi2}).
If (i) holds, then 
$\beta(G,c,\kappa)=0=g(G,c,\kappa)$.
If (ii) holds, and $c_0$ is the value of $c$ on $V(G)$, 
then (\ref{ewpi2}) implies 
$$\beta(G,c,\kappa)\leq g(G,c,\kappa)
=\sum\limits_{i=1}^n c(u_i)\cdot \frac{d_G(u_i)}{2}=c_0\cdot m,$$
where $m$ denotes the number of edges of $G$.
Furthermore, if $\iota:V(G)\to\mathbb{Z}_{\geq 0}$ is such that 
$V(G)$ is $\kappa'$-degenerate for $\kappa'=\kappa+\iota$,
and $u_1,\ldots,u_n$ is a linear ordering of the vertices of $G$
such that $u_i$ has at most $\kappa'(u_i)$ neighbors in $\{ u_j:j\in [i-1]\}$ 
for every $i\in [n]$,
then $\kappa(u_i)=0$ implies 
$\iota(u_i)\geq |N_G(u_i)\cap \{ u_j:j\in [i-1]\}|$ for every $i\in [n]$, 
and, hence,
$$\sum\limits_{i=1}^n c(u_i)\cdot \iota(u_i)
\geq c_0\sum\limits_{i=1}^n |N_G(u_i)\cap \{ u_j:j\in [i-1]\}|
=c_0\cdot m.$$
Altogether, we obtain 
$\beta(G,c,\kappa)=m=g(G,c,\kappa)$.
Finally, if (iii) holds, $G$ has order $n$,
$c(u)=c_0$ and $\kappa(u)=\kappa_0$ for every vertex $u$ of $G$,
then $\beta(G,c,\kappa)=c_0\big(1+2+\ldots+(n-1-\kappa_0)\big)=g(G,c,\kappa)$.
Hence, if (i), (ii), or (iii) hold, then (\ref{ewpi2}) holds with equality.
We call a triple $(G,c,\kappa)$ {\it extremal} 
if $\beta(G,c,\kappa)=g(G,c,\kappa)$.
Now, let $(G,c,\kappa)$ be extremal.
Note that we do not yet assume that $G$ is connected.

\setcounter{claim}{0}

\begin{claim}\label{claim1b}
For every vertex $u$ of $G$,
\begin{eqnarray*}
c(u)(d_G(u)-\kappa(u)) & = & 
\frac{c(u)(d_G(u)-\kappa(u))(d_G(u)-\kappa(u)+1)}{2(d_G(u)+1)}\\
&& +\sum\limits_{v\in N_G(u)}
\frac{c(v)(d_G(v)-\kappa(v))(d_G(v)-\kappa(v)+1)}{2(d_G(v)+1)}\\
&& -\sum\limits_{v\in N_G(u)}
\frac{c(v)(d_G(v)-\kappa(v)-1)(d_G(v)-\kappa(v))}{2d_G(v)}.
\end{eqnarray*}
\end{claim}
\begin{proof}[Proof of Claim \ref{claim1b}]
Arguing as in the proof of Claim \ref{claim1} in the proof of Theorem \ref{theorem1},
we obtain that the sum of 
the differences of the left hand side and the right hand side 
of the expression in the statement of the claim equals $0$.
Hence, if the statement of the claim does not hold,
then there is a vertex $u'$ of $G$ with 

\begin{eqnarray}
c(u')(d_G(u')-\kappa(u')) & < & \nonumber
\frac{c(u')(d_G(u')-\kappa(u'))(d_G(u')-\kappa(u')+1)}{2(d_G(u')+1)}\\
&& +\sum\limits_{v\in N_G(u')}\label{eb0b}
\frac{c(v)(d_G(v)-\kappa(v))(d_G(v)-\kappa(v)+1)}{2(d_G(v)+1)}\\
&& -\sum\limits_{v\in N_G(u')}\nonumber
\frac{c(v)(d_G(v)-\kappa(v)-1)(d_G(v)-\kappa(v))}{2d_G(v)}.
\end{eqnarray}
Let 
\begin{eqnarray}
N' & = & \big\{ v\in N_G(u'):\kappa(v)=d_G(v)\big\},\nonumber\\
V' & = & V(G)\setminus \{ u'\},\nonumber\\
G' & = & G\left[V' \right],\label{eb3b}\\
c' & = & c\mid_{V'}\mbox{, and}\nonumber\\
\kappa' &:& V'\to \mathbb{Z}:
v\mapsto 
\begin{cases}
\kappa(v)-1 &\mbox{, if $v\in N'$, and}\\
\kappa(v) & \mbox{, if $v\in V'\setminus N'$.}
\end{cases}\nonumber
\end{eqnarray}
Note that, unlike in the proof of Theorem \ref{theorem1},
the vertices in $N'$ still belong to $G'$.
By construction, $0\leq \kappa'(u)\leq d_{G'}(u)$ for every vertex $u$ of $G'$.
This implies the contradiction
\begin{eqnarray*}
\beta(G,c,\kappa) 
& \leq & \beta(G',c',\kappa')+c(u')(d_G(u')-\kappa(u'))\\
& \stackrel{(\ref{ewpi2})}{\leq} & 
g(G',c',\kappa')+c(u')(d_G(u')-\kappa(u'))\\
& \stackrel{(\ref{eb0b})}{<} &
g(G',c',\kappa')
+\frac{c(u')(d_G(u')-\kappa(u'))(d_G(u')-\kappa(u')+1)}{2(d_G(u')+1)}\\
&& +\sum\limits_{v\in N_G(u')}
\frac{c(v)(d_G(v)-\kappa(v))(d_G(v)-\kappa(v)+1)}{2(d_G(v)+1)}\\
&& -\sum\limits_{v\in N_G(u')}
\frac{c(v)(d_G(v)-\kappa(v)-1)(d_G(v)-\kappa(v))}{2d_G(v)}\\
& = & g(G,c,\kappa),
\end{eqnarray*}
which completes the proof of the claim.

Unlike in the proof of Theorem \ref{theorem1},
the final equality within the above inequality chain always holds with equality.
In fact, 
\begin{itemize}
\item a vertex $w$ in $V'\setminus N_G(u')$ contributes exactly the same to 
$g(G',c',\kappa')$ and $g(G,c,\kappa)$,
\item a vertex $v$ in $N_G(u')\setminus N'$ 
contributes 

$\frac{c(v)(d_G(v)-\kappa(v))(d_G(v)-\kappa(v)+1)}{2(d_G(v)+1)}$
to $g(G,c,\kappa)$
and 

$\frac{c(v)(d_G(v)-\kappa(v)-1)(d_G(v)-\kappa(v))}{2d_G(v)}$
to $g(G',c',\kappa')$, 
and, 
\item a vertex $v$ in $N'$ 
contributes 

$0=\frac{c(v)(d_G(v)-\kappa(v))(d_G(v)-\kappa(v)+1)}{2(d_G(v)+1)}$
to $g(G,c,\kappa)$and 

$0=\frac{c(v)(d_G(v)-\kappa(v)-1)(d_G(v)-\kappa(v))}{2d_G(v)}$
to $g(G',c',\kappa')$.
\end{itemize}
\end{proof}
We say that a vertex $u$ of $G$ is {\it terminal}
if there is a function $\iota:V(G)\to \mathbb{Z}_{\geq 0}$
such that 
\begin{itemize}
\item $\iota(V(G))=\beta(G,c,\kappa)$,
\item $V(G)$ is $(\kappa+\iota)$-degenerate, and 
\item $d_G(u)\leq \kappa(u)+\iota(u)$.
\end{itemize}
Note that $\kappa(u)\leq d_G(u)$ and the optimality of $\iota$ imply
that $\iota(u)=d_G(u)-\kappa(u)$.
\begin{claim}\label{claim2b}
Let $u'$ be any vertex of $G$,
and let $N'$, $V'$, $G'$, $c'$, and $\kappa'$ 
be as in (\ref{eb3b}).

The vertex $u'$ is terminal, and
$(G',c',\kappa')$ is extremal.
\end{claim}
\begin{proof}[Proof of Claim \ref{claim2b}]
We obtain
\begin{eqnarray*}
\beta(G,c,\kappa) 
& \leq & \beta(G',c',\kappa')+c(u')(d_G(u')-\kappa(u'))\\
& \stackrel{(\ref{ewpi2})}{\leq} & 
g(G',c',\kappa')+c(u')(d_G(u')-\kappa(u'))\\
& \stackrel{\rm Claim\,\,\ref{claim1b}}{=} &
g(G',c',\kappa')
+\frac{c(u')(d_G(u')-\kappa(u'))(d_G(u')-\kappa(u')+1)}{2(d_G(u')+1)}\\
&& +\sum\limits_{v\in N_G(u')}
\frac{c(v)(d_G(v)-\kappa(v))(d_G(v)-\kappa(v)+1)}{2(d_G(v)+1)}\\
&& -\sum\limits_{v\in N_G(u')}
\frac{c(v)(d_G(v)-\kappa(v)-1)(d_G(v)-\kappa(v))}{2d_G(v)}\\
& = & g(G,c,\kappa),
\end{eqnarray*}
Since $\beta(G,c,\kappa)=g(G,c,\kappa)$,
equality holds throughout the above inequality chain.
Since $\beta(G,c,\kappa)=\beta(G',c',\kappa')+
c(u')(d_G(u')-\kappa(u'))$,
it follows that $u'$ is terminal.
Since $\beta(G',c',\kappa')=g(G',c',\kappa')$,
it follows that $(G',c',\kappa')$ is extremal.
\end{proof}

\begin{claim}\label{claim3b}
If $G$ is connected, then there is some $g_0\in\mathbb{R}_{\geq 0}$
with 
\begin{eqnarray}\label{e3b}
g_0=\frac{c(u)(d_G(u)-\kappa(u))(d_G(u)+\kappa(u)+1)}{2d_G(u)(d_G(u)+1)}
\end{eqnarray}
for every vertex $u$ of $G$.
\end{claim}
\begin{proof}[Proof of Claim \ref{claim3b}]
If $h(u)$ denotes the right hand side of (\ref{e3b}),
then Claim \ref{claim1} implies that 
$h(u)=\frac{1}{d_G(u)}\sum_{v\in N_G(u)}h(v)$ for every vertex $u$ of $G$,
that is, the $h$-value of every vertex equals 
the average $h$-value of its neighbors.
Since $G$ is connected, it follows that $h$ is constant within $V(G)$.
\end{proof}
We have shown Claims \ref{claim1b}, \ref{claim2b}, and \ref{claim3b}
for every extremal triple $(G,c,\kappa)$.
For the rest of the proof, we proceed by contradiction,
and assume that the extremal triple $(G,c,\kappa)$
is a counterexample to the statement of the theorem
such that the order $n$ of $G$ is minimum.
Trivially, we have $n\geq 2$.
If $g_0=0$, then $c(u)>0$ implies $\kappa(u)=d_G(u)$ 
for every vertex $u$ of $G$,
that is, (i) holds.
By the choice of $(G,c,\kappa)$, we obtain $g_0>0$, 
which implies $\kappa(u)<d_G(u)$ for every vertex $u$ of $G$.
Unlike in the proof of Theorem \ref{theorem1},
this implies that,
for any vertex $u'$ of $G$,
and $N'$, $V'$, and $\kappa'$ as in (\ref{eb3b}),
we have $N'=\emptyset$ and $\kappa'(u)=\kappa(u)$ for every vertex $u$ in $V'$.
Furthermore, if $n=2$, then this implies $\kappa(u)=0$ for every vertex $u$ of $G$, 
which, by (\ref{e3b}), implies that $c$ is constant,
that is, (iii) holds.
Hence, the choice of $(G,c,\kappa)$ implies $n\geq 3$.

In order to complete the proof using a similar approach 
as in the proof of Theorem \ref{theorem1}, 
we first need to handle the situation corresponding to (ii),
which leads to the following additional claim.

\begin{claim}\label{claim4a}
$\kappa(u)>0$ for every vertex $u$ of $G$.
\end{claim}
\begin{proof}[Proof of Claim \ref{claim4a}]
Suppose, for a contradiction, that $\kappa(x')=0$ for some vertex $x'$ of $G$.
Let $u'$ be a vertex of $G$ such that $u'$ is distinct from $x'$, and $G-u'$ is connected.
Let $N'$, $V'$, $G'$, $c'$, and $\kappa'$ be as in (\ref{eb3b}).
By Claim \ref{claim2b},
the triple $(G',c',\kappa')$ is extremal.
Since $G$ has at least $3$ vertices, we obtain $d_{G'}(x')>0$.
Since $\kappa'(x')=0$, 
the choice of $(G,c,\kappa)$ implies that $(G',c',\kappa')$ is as in (ii),
that is, $c'$ is constant on $V'$, and $\kappa'(u)=0$ for every vertex $u$ of $G'$.
Let $u''$ be a vertex of $G$ such that $u''$ is distinct from $u'$, and $G-u''$ is connected.
Let $N''$, $V''$, $G''$, $c''$, and $\kappa''$ be defined analogously as in (\ref{eb3b}).
Since $G$ has at least $3$ vertices, 
we have $\kappa''(x'')=0$ for some vertex $x''$ of $G''$.
Arguing as above, we obtain that 
$(G'',c'',\kappa'')$ is as in (ii),
that is, $c''$ is constant on $V''$, and $\kappa''(u)=0$ for every vertex $u$ of $G''$.
Since $V(G)=V'\cup V''$, 
it follows that 
$c$ is constant on $V(G)$, and $\kappa(u)=0$ for every vertex $u$ of $G$,
that is, (ii) holds.
This contradiction completes the proof of the claim.
\end{proof}
Now, we can proceed similarly as in the proof of Theorem \ref{theorem1}.

\begin{claim}\label{claim4b}
$G$ is a clique.
\end{claim}
\begin{proof}[Proof of Claim \ref{claim4b}]
Suppose, for a contradiction, that $G$ is not a clique.
This implies that $G$ has a vertex $u'$
such that $G'=G-u'$ is connected, and $u'$ is not universal.
Let $N'$, $V'$, $G'$, $c'$, and $\kappa'$ be as in (\ref{eb3b}).
Let $v'$ be a non-neighbor of $u'$ in $V'$.
Since $G$ is connected,
the vertex $u'$ has a neighbor $w'$ in $V'$.
By Claim \ref{claim2b},
the triple $(G',c',\kappa')$ is extremal,
which, by Claim \ref{claim3b}, implies
the existence of some $g_0'\in\mathbb{R}_{\geq 0}$
with 
\begin{eqnarray}\label{e3bprime}
g_0'=\frac{c(u)(d_{G'}(u)-\kappa(u))(d_{G'}(u)+\kappa(u)+1)}{2d_{G'}(u)(d_{G'}(u)+1)}
\end{eqnarray}
for every vertex $u$ of $G'$.
Now, we obtain 
\begin{eqnarray*}
g_0-g_0' & 
\stackrel{(\ref{e3b}),(\ref{e3bprime})}{=} & 
\frac{c(w')(d_G(w')-\kappa(w'))(d_G(w')+\kappa(w')+1)}{2d_G(w')(d_G(w')+1)}\\
&&-
\frac{c(w')\big((d_G(w')-1)-\kappa(w')\big)\big((d_G(w')+1)+\kappa(w')+1\big)}{2(d_G(w')-1)d_G(w')}\\
& = & 
\frac{c(w')\kappa(w')(\kappa(w')+1)}{d_G(w')(d_G(w')^2-1)}\\
& \stackrel{\rm Claim\,\,\ref{claim4a}}{>} & 0,\mbox{ and }\\[3mm]
g_0-g_0' 
&\stackrel{(\ref{e3b}),(\ref{e3bprime})}{=} & 
\frac{c(v')(d_G(v')-\kappa(v'))(d_G(v')+\kappa(v')+1)}{2d_G(v')(d_G(v')+1)}\\
&&-
\frac{c(v')(d_G(v')-\kappa(v'))(d_G(v')+\kappa(v')+1)}{2d_G(v')(d_G(v')+1)}\\
& = & 0.
\end{eqnarray*}
This contradiction completes the proof of the claim.

Note that we needed $\kappa(w')>0$,
that is, Claim \ref{claim4a}, for that contradiction.
\end{proof}
We are now in a position to derive a final contradiction.
Let $u'$ be any vertex of $G$,
and let $N'$, $V'$, $G'$, $c'$, and $\kappa'$ 
be as in (\ref{eb3b}).
By Claims \ref{claim2b} and \ref{claim4b}, 
$G'$ is complete and $(G',c',\kappa')$ is extremal.
By the choice of $(G,c,\kappa)$, 
this implies that $\kappa'$ is constant on $V'$;
regardless which of (i), (ii), or (iii) applies.
By the definition of $\kappa'$,
this implies that $\kappa$ is constant on $V(G)\setminus \{ u'\}$.
Since $u'$ was an arbitrary vertex of $G$, 
and $G$ has at least $3$ vertices,
this implies that $\kappa$ is constant on $V(G)$.
Since $g_0>0$, Claim \ref{claim3b} implies that $c$ is constant on $V(G)$,
that is, (ii) holds.
This final contradiction completes the proof. 
\end{proof}

\section{Conclusion}

There are versions of (\ref{ecgrv}) and (\ref{ewpi2})
that apply to functions $c$ and $\kappa$,
where $c$ is allowed to assume values that are less or equal to $0$,
and $\kappa$ is allowed to assume negative values.
It seems not too difficult --- yet slightly tedious ---
to extend Theorems \ref{theorem1} and \ref{theorem2}
in order to incorporate these cases.
In view of the extremal graphs,
there are several natural additional assumptions 
that one may impose on $G$ in order to improve (\ref{ecgrv}) and (\ref{ewpi2}).
In view of similar research for the independence number,
one may consider connectivity \cite{hara},  
triangle-freeness \cite{sh}, or 
local irregularity \cite{se} (cf.~\cite{hamo} for a corrected proof).

\end{document}